\documentclass[reqno]{amsart}
\usepackage{amsmath,amsfonts,amssymb,graphics,graphicx,latexsym,amscd, subfigure}

\newtheorem{thm}{Theorem}[section]

\newtheorem{prop}[thm]{Proposition}

\newtheorem{lemma}[thm]{Lemma}
\newtheorem{defn}[thm]{Definition}
\newtheorem{rem}[thm]{Remark}
\newtheorem{rems}[thm]{Remarks}

\newtheorem{exa}[thm]{Example}

\newcommand{\bbR}{\mathbb{R}}

\newcommand{\bbT}{\mathbb{T}}
\newcommand{\bbC}{\mathbb{C}}
\newcommand{\bbD}{\mathbb{D}}
\newcommand{\bbZ}{\mathbb{Z}}
\newcommand{\bbH}{\mathbb{H}}
\newcommand{\ip}{\cdot}

\newcommand{\mV}{{\mathcal V}}
\newcommand{\mU}{{\mathcal U}}
\newcommand{\bz}{{\bf z}}
\newcommand{\bx}{{\bf x}}

\newcommand{\R}{\mathbb R}

\newcommand{\al}{{\alpha}}

\DeclareMathOperator{\Real}{Re}
\DeclareMathOperator{\Imag}{Im}
\DeclareMathOperator{\spa}{Span}
\DeclareMathOperator{\Hess}{Hess}

\newcommand{\ov}{\overline}
\newcommand{\op}{{\overline{\partial}}}
\newcommand{\p}{{\partial}}
\newcommand{\cp}{{\mathbb{CP}}}

\begin{document}
\title[Uniqueness among scalar-flat K\"ahler toric metrics]{Uniqueness among scalar-flat K\"ahler metrics on non-compact toric $4$-manifolds}
\author{Rosa Sena-Dias}
\address{Centro de An\'{a}lise Matem\'{a}tica, Geometria e Sistemas Din\^{a}micos, Departamento de Matem\'atica, Instituto Superior T\'ecnico}
\email{rsenadias@math.ist.utl.pt}
\thanks{This work was partially supported by FCT/Portugal through project PTDC/MAT-GEO/1608/2014}
\begin{abstract}
In \cite{as}, the authors construct two distinct families of scalar-flat K\"ahler non-compact toric metrics using Donaldson's rephrasing of Joyce's construction in action-angle coordinates. In this paper and using the same set-up, we show that these are the only J-complete scalar-flat K\"ahler metrics on any given strictly unbounded toric surface. We also show that the asymptotic behaviour of such a metric determines it uniquely.
\end{abstract}

\maketitle
\section{Introduction}
Based on work of Donaldson's (see \cite{d1}) which in turn builds on an important construction of Joyce giving local models for  scalar-flat K\"ahler metrics with torus symmetry, Abreu and the author have constructed two distinct types of scalar-flat K\"ahler toric metrics on strictly unbounded toric symplectic $4$-folds. Namely
\begin{itemize}
\item An ALE scalar-flat K\"ahler toric metric which had been previously written down in \cite{cs}.
\item A family of so-called Donaldson generalised Taub-NUT metrics which are all complete scalar-flat K\"ahler metrics. Some of these metrics are Ricci-flat and were previously known (see \cite{l1}). Others are actually not Ricci-flat and are new.
\end{itemize}
The construction uses action-angle coordinates and gives the so-called symplectic potentials of the metrics. In \cite{d1}, Donaldson explains how to establish a local correspondence between 
\begin{enumerate}
\item Solutions to a non-linear PDE describing symplectic potentials of scalar-flat K\"ahler metrics.
\item Pairs of solutions to a familiar linear PDE describing axi-symmetric harmonic functions on $\bbR^3$ which give a local diffeomorphism from $\bbR^2$ to $\bbR^2$.
\end{enumerate}
\cite{as} uses only one side of this correspondence namely that a pair of axi-symmetric harmonic functions on $\bbR^3$ determines the local symplectic potential of a scalar-flat K\"ahler metric. It is natural to ask if the above correspondence can be used to prove some uniqueness results for scalar-flat K\"ahler metrics on strictly unbounded symplectic $4$-folds. The goal of this paper is to show that indeed one can use the correspondence to prove that the metrics constructed in \cite{as} are essentially the only scalar-flat K\"ahler metrics one can construct on strictly unbounded toric $4$-folds. More precisely we prove the following theorems.
\begin{thm}\label{unique}
Let $(X,\omega)$ be a strictly unbounded toric $4$-manifold and $J$ a complete, compatible complex structure which is torus invariant. Let $g$ be the corresponding toric metric determined by $\omega$ and $J$. If $g$ is scalar-flat then either
\begin{enumerate}
\item $g$ is the ALE toric scalar-flat K\"ahler metrics on $X$ or
\item $g$ is a Donaldson generalised Taub-NUT metric.
\end{enumerate}
\end{thm}
The theorem ensures that the metrics constructed in \cite{as}, are the only possible complete scalar-flat K\"ahler toric metrics. The metrics in \cite{as} which are not ALE are referred to as Donaldson generalised Taub-NUT metrics. We also prove that the asymptotic behaviour for a toric scalar-flat K\"ahler metric on a strictly unbounded toric $4$-fold determines it uniquely. 
\begin{thm}\label{unique_prescibed}
Let $(X,\omega)$ be a strictly unbounded toric $4$-manifold and $J$ a complete, compatible complex structure which is torus invariant. Let $g$ be the corresponding toric metric determined by $\omega$ and $J$. If $g$ is scalar-flat and asymptotic to one of the Donaldson generalised Taub-NUT metrics from \cite{as} with parameter $\nu,$ then $g$ is the Donaldson generalised Taub-NUT metric with parameter $\nu$ on $X$.
\end{thm}
We will give a more precise statement of the above theorem ahead. We will in particular explain what we mean by ``asymptotic to a Taub-NUT metric". 

The idea of using Donaldson's action-angle coordinates version of Joyce's construction to tackle uniqueness questions first appeared in \cite{w3}. Wright even suggests a strategy to approach the question and proves some partial results. This is essentially the strategy we pursue here. Another approach to this question is developed by Weber in \cite{w}, using very different techniques.

\begin{rems}
A strictly unbounded toric $4$-manifold is one whose moment polytope is unbounded and whose unbounded edges are not parallel. The notion appears in \cite{as} but we will give a more precise definition ahead.

Part of Theorem \ref{unique_prescibed}, namely the part that corresponds to ALE metrics and $\nu=0$, was proved in \cite{w1} using Twistors. Our asymptotic behaviour requirement is slightly less restrictive. We prove that the metric is the ALE metric assuming only that the metric is asymptotic to the Euclidean metric when restricted to the polytope.
\end{rems}
\vspace{.2cm}
\noindent \textbf{Acknowledgments:} I would to thank Simon Donaldson for his support and Gustavo Granja for interesting conversations.

\section{Unbounded toric $4$-manifolds}

In this section we give a quick review of some facts about unbounded toric manifolds. For more details see \cite{as}.
\begin{defn}
A symplectic $4$-manifold $(X,\omega)$ is said to be toric if it admits an effective Hamiltonian $\bbT^2$-action whose moment map is proper.
\end{defn}
In this setting the moment map image $P$ is the closure of a convex polygonal region in $\bbR^2$ of the form
\[
P=\{x\in \bbR^2: x\ip \nu_i-\lambda_i> 0, \,\, i=1\cdots d\}
\]
where 
\begin{itemize}
\item $d$ is an integer (the number of facets of $P$), 
\item for each $i$, $\nu_i$ is a vector in $\bbR^2$ the primitive interior normal to the $i$th facet,
\item $\lambda_i$ is a real number. 
\end{itemize}
The $\nu_i$'s satisfy the Delzant condition. Namely,
\[
det(\nu_{i},\nu_{i+1})=-1,
\]
for all $i=1,\cdots d-1$. It is now easy to define an unbounded toric manifold.
\begin{defn}
A symplectic toric $4$-manifold $(X,\omega)$ is said to be unbounded if its moment map image is unbounded. It is said to be strictly unbounded if it is unbounded and the unbounded edges in its moment map image are not parallel.
\end{defn}
It is a important result of Guillemin's that every toric symplectic manifold actually admits a K\"ahler structure. 
\begin{thm}[Guillemin]
Let $(X,\omega)$ be a symplectic toric manifold with moment map image $\bar{P}$. Then $(X,\omega)$ admits an integrable, compatible, torus invariant complex structure. This structure is completely determined by $P$. 
\end{thm}
We will call the resulting metric Guillemin's metric and denote the K\"ahler structure by $(X,\omega_P,J_P,g_P)$. We will use integrable, compatible, torus invariant complex structures to parametrize Riemannian metrics on $X$. The advantage of toric manifolds is that there is a particularly nice way to parametrize such complex structure as we will see ahead. We will only be concerned with metrics that come from integrable, compatible, torus invariant complex structures which are equivariantly biholomorphic to $J_P$. 
\begin{defn}
Let $(X,\omega)$ be a symplectic toric manifold with moment map image $P$. We say that an integrable, compatible, torus invariant complex structure is complete if it is equivariantly biholomorphic to $J_P$.
\end{defn}
At this point we would like to give some important examples. We start with a very simple and familiar one.
\begin{exa}
Let $X=\bbC^2=\bbR^4$ with coordinates $(z_1,z_2)$ with symplectic form
\[
\omega=dz_1\wedge d\bar{z}_1+dz_2\wedge d\bar{z}_2.
\]
Let $\bbT^2=S^1\times S^1$ act via the usual $S^1$-action on $\bbC$. Then it is easy to see that $X$ is toric and the moment map of the torus action is
\[
\phi(z_1,z_2)=(|z_1|^2,|z_2|^2),
\]
so that the moment map image is the closure of $P=(\bbR^+)^2$. This has two edges both of which are unbounded.
\end{exa}
Another important example is the following.
\begin{exa}
Let $\Gamma$ be a finite cyclic subgroup of $U(2)$ whose action on $\bbC^2$ has finite isolated singularities. Then the orbifold $\bbC^2/\Gamma$ is toric (although we have not defined toric orbifold it is straightforward to generalize the above definition for manifolds). More importantly this orbifold admits a minimal resolution $X_\Gamma$ which is itself a toric $4$-manifold. There is a complex structure on $X_\Gamma$ coming from the usual complex structure on $\bbC^2$. One can check that this is equivariantly biholomorphic to the Guillemin complex structure.
\end{exa}
Once we have a symplectic toric manifold, we can preform symplectic blow-ups of any finite number of fixed points  to obtain other symplectic toric manifolds. It is well known that such an operation corresponds to ``corner chopping" on the moment map image. 

These examples essentially give all possible strictly unbounded toric $4$-manifolds up to equivariant biholomorphism. In \cite{as}, we prove the following proposition.
\begin{prop}
Let $X$ be a strictly unbounded toric $4$-manifold. Then $X$ is equivariantly biholomorphic to a finite iterated blow-up of some $X_\Gamma$ for some subgroup $\Gamma$ of $U(2)$.
\end{prop}
\section{Symplectic potential}
Symplectic potentials are the way in which we parametrize toric K\"ahler metrics on toric manifolds. In this section we briefly explain what these are and why they are useful. For more details see \cite{as} and \cite{a1}. Let $(X,\omega)$ be a symplectic toric $4$-manifold. Denote its moment map by $\phi:X\rightarrow \bbR^2$ and its moment map image $\phi(X)$ by $\bar{P}$ where, 
\[
P=\{x\in \bbR^2: x\ip \nu_i-\lambda_i> 0, \,\, i=1\cdots d\}.
\]
Also, write $(x_1,x_2)$ for the coordinates of the moment map. Identify $\bbT^2$ with $S^1\times S^1$ and let $(\theta_1,\theta_2)$ denote the angle coordinates on each $S^1$ factor. Then $(x,\theta)$ are Darboux coordinates for $\omega=\omega_P$ i.e.
\[
\omega_P=dx_1\wedge d\theta_1+dx_2\wedge d\theta_2.
\]
The important point here is that in these coordinates any $J$ which is a complete, integrable, compatible, torus invariant complex structure is given by
\[
\begin{bmatrix}
\phantom{-}0\ \  & \vdots & -\Hess(u)^{-1}  \\
\hdotsfor{3} \\
\phantom{-}\Hess(u)\ \  & \vdots & 0\,
\end{bmatrix}
\]
for some function $u$ which is called \emph{symplectic potential}. One can also write the Riemannian metric in terms of this potential: 
\[
\begin{bmatrix}
\phantom{-}\Hess(u)\ \  & \vdots & 0  \\
\hdotsfor{3} \\
\phantom{-}0\ \  & \vdots & \Hess(u)^{-1}\,
\end{bmatrix}.
\]
\begin{exa}
Guillemin calculated the symplectic potential $u_P$ for the Guillemin metric:
\[
u_P=\sum_{i=1}^d l_i\log l_i-l_i
\]
where $l_i(x)= x\ip \nu_i-\lambda_i$.
\end{exa}
This example is particularly important because of the following proposition from \cite{a2}:
\begin{prop}[Abreu]\label{boundary_u}
Let $X$ be a symplectic toric $4$-manifold with moment map image $\bar{P}$. Let $J$ be any complete, integrable, compatible, torus invariant complex structure with symplectic potential $u$. Then $u-u_P$ is a smooth function on a neighbourhood of $\bar{P}$.
\end{prop}
In \cite{a1} Abreu wrote down a formula for scalar curvature in terms of symplectic potential.
\begin{thm}[Abreu]
The scalar curvature of a torus invariant metric determined by the symplectic potential $u$ is given by
\[
s = -\frac{1}{2} \sum_{j,k} \frac{\partial^2 u^{jk}}{\partial x_i\partial x_j},
\]
where $u^{ij}$ denote the entries of the matrix $\Hess(u)^{-1}$.
\end{thm}
Hence the relevant PDE for symplectic potentials of scalar-flat K\"ahler toric metrics is 
\begin{equation}\label{nlPDE}
 \frac{\partial^2 u^{jk}}{\partial x_i\partial x_j}=0
 \end{equation}
where we have used Einstein's convention for sums of repeated indices. This is a non-linear second order PDE. 
\section{The real form of a toric K\"ahler manifold}\label{real_form}
In this paper we will make use the real sub-manifold associated with $(X,\omega)$ and we start by quickly reviewing its construction. This is tightly linked to the construction of the toric manifold itself as described in \cite{g}. In \cite{dls} the real manifold is further described and referred to as real form of the toric K\"ahler manifold. 

We will assume that $P$ is normalised so that $\nu_1=(0,1)$, $\nu_2=(1,0)$ and the corresponding edges meet at zero. That is, $P$ is normalised so as to be standard at the vertex $0$. We can always assume this as $P$ is only defined up to $SL(2,\bbZ)$. We write the other normals of $P$ as $\nu_i =(\beta_i,\alpha_i), \, i=1,\cdots, d$ and notice that
$$
\alpha_i\nu_1+\beta_i\nu_2-\nu_i=0, \, i=3,\cdots d.
$$
The normals of the moment polytope $P$ determine a short exact sequence,
$$
\{0\}\rightarrow N\rightarrow \bbC^d\rightarrow \bbC^2\rightarrow \{0\}
$$
where the map $\bbC^d\rightarrow \bbC^2$ is the linear mapping sending $e_i$ to $\nu_i$ for $i=1,\cdots d.$ Here $\{e_1,\cdots e_d\}$ is the standard basis for $\bbC^d$. The subspace $N$ is given by
$$
N=\spa \{\alpha_i e_1+\beta_i e_2-e_i, i=3,\cdots d\},
$$
or $N=\{0\}$ if $d=2$. Then we may define 
$$
X=\{(z_1,\cdots, z_d):\alpha_i|z_1|^2+\beta_i|z_2|^2-|z_i|^2=\lambda_i, i=3,\cdots, d \}/N,
$$ 
which admits a $\bbT^d/N$-action. We can, by making a choice, identify $\bbT^d/N$ with $\bbT^2$. Let us say that we choose this identification so that
$$
(e^{i\theta_1},e^{i\theta_2})[z_1,z_2, \cdots,z_d]=[e^{i\theta_1}z_1, e^{i\theta_2}z_2, z_3,\cdots z_d].
$$
Given this choice the moment map $\phi: X\rightarrow \bbR^2$ is given by $\phi([z_1,z_2, \cdots,z_d])=(|z_1|^2,|z_2|^2)$. Complex conjugation on $\bbC^2$ descends to a function on $X$ and its fixed point set is the real sub-manifold of $X$ denoted by $X_\bbR$. It is given by
$$
X_\bbR=\{(\bx_1,\cdots, \bx_d):\alpha_i\bx_1^2+\beta_i\bx_2^2-\bx_i^2=\lambda_i, i=3,\cdots, d \}/N_\bbR
$$
where $N_\bbR=N\cap \bbR^d.$ The restriction of $\phi$ to $X_\bbR$ is denoted by $\phi_\bbR$. The map $\phi_\bbR:X_\bbR\rightarrow \bbR^2$ is a $4$ to $1$ branched cover branched along $\phi_\bbR^{-1}(\partial P)$ and $\breve{X}_{\bbR}=\phi_\bbR^{-1}(P)$ can be written as the disjoint union of four open sets $P_j,\, j=0,\cdots 3$. For each such open subset there is a map $\sigma_j:P_j\rightarrow P_0$ given by the action of one of the four elements of $\bbT^2$ preserving $X_\bbR$ namely $(\pm 1,\pm1)$.  Any toric K\"ahler metric $g$ on $X$ induces a metric $g_\bbR$ on $X_\bbR$ and on the open set $P_0$. This metric can be identified with $u_{ij}dx_i\otimes dx_j$ more precisely
$$
g_r={\phi_\bbR}_*( g_\bbR)=u_{ij}dx_i\otimes dx_j.
$$
Because $\bbT^2$ acts by isometries on $X$, each $\sigma_j$ is an isometry for $g_\bbR$. 

We finish this section by noting that unless $P$ is a rectangle or $(\bbR^+)^2$, $X_\bbR$ is not orientable. It admits an orientable double cover $X^{o}_\bbR\rightarrow X_\bbR$ which can be thought of as a gluing of eight copies of $P$, $\tilde{P}_j^k$ $j=0,\cdots 3$ and $k=0,1$ where each $\tilde{P}_j^k$ is the pre-image under the double cover of $P_j$ discussed above. The metric $g_\bbR$ induces a metric on $X^{o}_\bbR$ and the $\sigma_i$ can be lifted as isometries. We will also need the lift of $\phi_\bbR$ to $X^{o}_\bbR$ which we denote by $\phi^o_\bbR:X^{o}_\bbR\rightarrow \bbR^2.$ Because $X^{o}_\bbR$ is orientable, a metric on it determines a complex structure via the Hodge star which we denote $J_\bbR$. The manifold $X_\bbR$ itself does not admit such a complex structure this is why we need to use $X^{o}_\bbR$ rather than $X_\bbR$. This structure also induces a complex structure on $P$, $J_r$.

\section{Joyce's construction in action-angle coordinates}\label{joyce_construction}
In \cite{d1} Donaldson translates Joyce's construction of local scalar-flat K\"ahler metrics with torus symmetries into the language of symplectic potentials and action-angle coordinates. In this section our aim is to describe Donaldson's construction. For more details and proofs see \cite{d1}. Consider the following linear PDE
\begin{equation}\label{PDE}
 \frac{\partial^2 \xi}{\partial H^2}+\frac{\partial^2\xi}{\partial r^2}+\frac{1}{r}\frac{\partial \xi}{\partial r }=0,
\end{equation}
on $\mathbb{H}=\{(H,r)\in\R^2:r > 0\}$. 
\begin{thm}[Donaldson,\cite{d1}]\label{donaldson's}
Let $\xi_1$ and $\xi_2$ be two solutions of equation (\ref{PDE}) on an open subset $B$ of  $\mathbb{H}$. Let
\begin{displaymath}
 \epsilon_1=r\left(\frac{\partial \xi_2}{\partial r} dH-\frac{\partial \xi_2}{\partial H} dr \right)
\end{displaymath}
and
\begin{displaymath}
 \epsilon_2=-r\left(\frac{\partial \xi_1}{\partial r} dH-\frac{\partial \xi_1}{\partial H} dr \right).
\end{displaymath}
Then these two $1$-forms are closed. Let $x_1$ and $x_2$ denote their primitives, well defined up to a constant. Then $(x_1,x_2)$ are local coordinates in $\mathbb{R}^2$. Let
\begin{displaymath}
 \epsilon=\xi_1 dx_1+\xi_2 dx_2.
\end{displaymath}
This $1$-form is also closed. Let $u$ be a primitive of $\epsilon$ and write $\xi=(\xi_1,\xi_2)$. Then, if $\det D\xi>0$, where
\[
D\xi=\begin{pmatrix}
\frac{\partial \xi_{1}}{\partial H} & \frac{\partial \xi_{1}}{\partial r}  \\
\frac{\partial \xi_{2}}{\partial H} & \frac{\partial \xi_{2}}{\partial r}  
\end{pmatrix},
\]    
the function $u$ is a local symplectic potential for some toric K\"ahler metric on $\mathbb{R}^4$ whose scalar curvature is $0$ i.e.
\begin{enumerate}\label{conditions_u}
\item $\Hess(u)$ is positive definite on the interior of the polytope and when restricted to $i$-th facet of the moment polytope $l_i^{ -1}(0)$ the non-singular part of $u$ i.e. $u-l_i\log (l_i)$ is convex on the interior of $l_i^{ -1}(0)$.
\item $u$ solves equation (\ref{nlPDE}).
\end{enumerate}
\end{thm}
The above construction is local but one can use it with the appropriate boundary conditions on the solutions of equation (\ref{PDE}) to construct global metrics.
In \cite{d1} Donaldson also explains that the above construction is reversible. That is, starting from an $u$ satisfying conditions (\ref{conditions_u}),
we can find $(H,r)$ in terms of $(x_1,x_2)$ as well as  a pair of solutions $(\xi_1,\xi_2)(H,r)$ of equation (\ref{PDE}) such that  $\det D\xi>0$. More precisely
let $r=(\det \Hess u)^{-1/2}$ and $H$ so that 
$$
\begin{aligned}
\frac{\partial H}{\partial x_1}&=&-\frac{u^{2j}}{r}\frac{\partial r}{\partial x_j}\\
\frac{\partial H}{\partial x_2}&=&\frac{u^{1j}}{r}\frac{\partial r}{\partial x_j}.\\
\end{aligned}
$$
Donaldson shows that $r$ and $H$ as functions on the polytope $P$ endowed with the metric $g_r=u_{ij}dx_i\otimes dx_j$ are harmonic and harmonic conjugates. The way we will want to interpret this is the following. The metric $g_r$ on $P$ induces a complex structure $J_r$ via its Hodge star. By definition this complex structure is such that $a \circ J_r=\star a$ for any $1$-form $a$ on $P$. We have $\star dr=-dH$. Therefore $$d(H+ir)\circ J_r=\star d(H+ir)=id(H+ir).$$ 

Set $\bz=H+ir$. This is a $J_r$-holomorphic local coordinate on $P$.



\section{The map $\bz=H+ir$}
The goal of this section of this section is to prove the following proposition.
\begin{prop}
Let $(X,\omega)$ be a strictly unbounded toric $4$-manifold and $J$ a complete, compatible complex structure which is torus invariant. Let $g$ be the corresponding toric metric determined by $\omega$ and $J$. Let $r=(\det \Hess u)^{-1/2}$ and $H$ so that 
$$
\begin{aligned}
\frac{\partial H}{\partial x_1}&=&-\frac{u^{2j}}{r}\frac{\partial r}{\partial x_j}\\
\frac{\partial H}{\partial x_2}&=&\frac{u^{1j}}{r}\frac{\partial r}{\partial x_j}.\\
\end{aligned}
$$
The map $\bz=H+ir: P\rightarrow \bbH$ is a bijection.
\end{prop}
We start by noticing that because the boundary behaviour of $u$ is determined on $\partial P,$ the behaviour of $\bz$ is also determined on $\partial P$ (for more details on this see \cite{s}). 
We can see that $\bz$ extends to $\partial P$ as a continuous function which we denote by $\tilde{\bz}$ and that because $r$ vanishes on the boundary, $\tilde{\bz}(\partial P)\subset \partial\bbH.$ It is not hard to see that $\tilde{\bz}_{|\partial P}:\partial P\rightarrow \partial \bbH$ is a bijection as it coincides with the extension of the map $\bz$ for the ALE metric which we can calculate explicitly.
\subsection{Injectivity of $\bz$}
We start by showing injectivity. 
\begin{lemma}
Let $w\in \partial P$ not a vertex of $P$ and $p$ one of the $4$ elements in $(\phi^o_\bbR)^{-1}(w)\subset X^{o}_\bbR$. There is a neighbourhood $\mV_p$ of $p$ in $X^o_\bbR$ such that $\bz\circ \phi^o_\bbR$ extends to $\mV_p$ as a holomorphic function for $J_\bbR.$
\end{lemma}
\begin{proof}
As before assume that $0\in \bar{P}$ and that the normals to the facets of $P$ meeting at $0$ are $(1,0)$ and $(0,1)$. There is a single pre-image of $w$ in $X_\bbR$ via $\phi_\bbR$ and two pre-images in $X_\bbR^o$ via $\phi^o_\bbR$. In the above description of $X_\bbR$ from section \ref{real_form} we have written $$\phi_\bbR([\bx_1,\cdots, \bx_d])=(|\bx_1|^2,|\bx_2|^2)$$ where
$$
[\bx_1,\cdots, \bx_d] \in X_\bbR=\{(\bx_1,\cdots, \bx_d):\alpha_i\bx_1^2+\beta_i\bx_2^2-\bx_i^2=\lambda_i, i=3,\cdots, d \}/N_\bbR.
$$
We may as well assume that $w\in l_1^{-1}(0)$ and is not zero. Let $P_j, j=0,\cdots 3$ be the open sets decomposing $\breve{X}_\bbR$ so that $p$ is on $\bar{P_0}\cap \bar{P_1}.$ We have 
$$
u(x,y)=x\log (x)+u_0(x,y)
$$
where $u_0$ is smooth in neighbourhood of $w$ in $P$. Since $r(x,y)=(\det (\Hess(u)))^{-1/2}$ 
$$
r(x,y)=\sqrt{x}\left(u_{0,22}+x\det(\Hess (u_0))\right)^{-1/2}.
$$ 
Because of the first condition on $u$ from (\ref{conditions_u}), $u_{0,22}(0,y)$ is non-vanishing near $w$. Therefore, on $P_0$
$$
r\circ\phi_\bbR([\bx_1,\cdots, \bx_d])={|\bx_1|}\left(u_{0,22}(\bx_1^2,\bx_2^2)+\bx_1^2\det(\Hess (u_0))(\bx_1^2,\bx_2^2)\right)^{-1/2}.
$$ 
The variable $\bx_1$ is not a priori well defined on $X_\bbR$ but we let $\bx_1:\mV_p\rightarrow \bbR$ be defined by
 $$
\bx_1([\bx_1,\cdots, \bx_d]))=
  \begin{cases}
  \;\sqrt{\bx_1^2} ,\, [\bx_1,\cdots, \bx_d]\in P_0\cap \mV_p\\
  -\sqrt{\bx_1^2} ,\, [\bx_1,\cdots, \bx_d]\in P_1\cap \mV_p.\\
 \end{cases}
 $$
This function is smooth in a neighbourhood of $\phi_\bbR^{-1}(w)$.  Set
$$
\tilde{r}([\bx_1,\cdots, \bx_d])={\bx_1}\left(u_{0,22}(\bx_1^2,\bx_2^2)+\bx_1^2\det(\Hess (u_0))(\bx_1^2,\bx_2^2)\right)^{-1/2}, \, [\bx_1,\cdots, \bx_d]\in \mV_p.
$$
This is then smooth and coincides with $r\circ \phi_\bbR$ on $P_0$ where defined. Because $r$ is harmonic on $P$ and $\phi_\bbR:P_0\rightarrow P$ is an isometry then $r\circ \phi_\bbR$ is harmonic on $P_0$ where defined. On the other hand $\tilde{r}_{|P_1}=\tilde{r}\circ {\sigma_1}_{|P_1}$ and $\sigma_1$ arises from the toric action and is therefore an isometry so that $\tilde{r}_{|P_1}$ is harmonic on $P_1$ where defined. We conclude that $\tilde{r}$ is harmonic. We now lift it to a neighbourhood $\mV_p$ of one of the preimages $p\in (\phi_\bbR^o)^{-1}(w)\subset X^o_\bbR,$ to get a harmonic function $r^o$. We can the define ${H^o}$ to be its harmonic conjugate coinciding with $H$ on $P_0$ and ${H^o}+i{r^o}$ is the extension we are seeking.
\end{proof}
\begin{exa}
To illustrate the above lemma consider the case when $P=(\bbR^+)^2$ corresponding to $X=\bbC^2$ and $X_\bbR=\bbR^2$. Assume $X$ is endowed with the flat metric. Then
$\bz(x,y)=x-y+2i\sqrt{xy}$. On the other hand $\phi(z_1,z_2)=(|z_1|^2,|z_2|^2)$ and $\phi_\bbR(x_1,x_2)=((x_1)^2,(x_2)^2)$ where $z_j=x_j+iy_j$ for $j=1,2$ are the coordinates in $\bbC^2$ and therefore $\bz\circ \phi_\bbR(x_1,x_2)=x_1^2-y_1^2+2i|{x_1x_2}|$ on $P=(\bbR^+)^2$. But this clearly extends as 
$$(x_1+ix_2)\mapsto (x_1+ix_2)^2$$
which is holomorphic.
\end{exa}
With this extension and given the fact $\tilde{\bz}$ is bijective on the boundary we will be able to prove that $\bz$ is injective. 
\begin{proof}[Proof of injectivity of $\bz$] Because $\bz:(P,J_r)\rightarrow (\bbH, J_0)$ is holomorphic it has a well defined notion of degree. We need to show that the degree of $\bz$ is $1$. Consider $w_0\in \partial \bbH$ and let $w\in \partial P$ be the single pre-image of $w_0$ via $\tilde{\bz}$. From the lemma there is an extension of $\bz$ to a neighbourhood of $p\in (\phi_\bbR^o)^{-1}(w)\subset X^o_\bbR,$ $\mV_p$. We can chose $p$ to be in the closure of $P^0_0$ and $P^0_1$.

Informally, the idea is that the number of pre-images of a point in $\bbH$ can be calculated via an integral around a loop enclosing all those pre-images.  By enlarging such a loop to enclose $w\in \partial P$ which we identify with $p,$ we can show that this is also the number of pre-images of $w_0$ which is known to be one. The details follow. 

Suppose $\mV_p$ is small enough to admit a complex chart $z:\mV_p\rightarrow \bbC$ and let $\epsilon$ be such that $B_\epsilon(w)\cap P\subset \phi^o_\bbR(\mV_p).$
We start by noticing that 
$$
\forall \epsilon \, \exists \delta: \, \bz^{-1}(B_\delta(w_0))\subset B_\epsilon(w).
$$
This is straightforward. Suppose it is not true. Then there is $\epsilon>0$ for which we can find a sequence $\{w_k\}\subset P$ with $|\bz(w_k)-w_0|\leq 1/k$ and $|w_k-w|>\epsilon$. This sequence is bounded in $P$ (otherwise infinity would be in the pre-image of $w_0\in \partial \bbH$) and admits a convergent subsequence in $\bar{P}.$ The limit $w_\infty$ satisfies $\tilde{\bz}(w_\infty)=w_0$ which as $\tilde{\bz}^{-1}(\partial H)=\partial P$ then forces $w_\infty=w.$ But this contradicts $|w_k-w|>\epsilon$.

Given $w_0'\in B_\delta(w_0)\cap \bbH$,  $\bz^{-1}(w_0')\subset B_\epsilon(w)$,
$$
\# \bz^{-1}(w_0')=\int_\gamma \frac{\frac{d\bz}{dz}(z)dz}{\bz(z)-w_0'},
$$
were $\gamma$ is a closed curve close to $\partial ( B_\epsilon(w_0)\cap P)$ and $z$ is the complex coordinate. Consider a lift of $\gamma$ to $P^0_0\subset X^o_\bbR$ containing $p$ in its closure. If $\epsilon$ is small enough, it is contained in a coordinate chart for $X^o_\bbR$. Now modify this lift slightly to enclose $p\in (\phi^o_\bbR)^{-1}(w_0)$ while remaining within $\mV_p$ and call this modification $\tilde{\gamma}$. Then for the extension $\zeta$ of $\bz\circ \phi^o_\bbR$ to $\mV_p$
$$
\# \zeta^{-1}(w_0')=\int_{\tilde{\gamma}} \frac{\frac{d\zeta}{dz}(z)dz}{\zeta(z)-w_0'},
$$
where $z$ denotes the complex coordinate on $\mV_p$. The quantity 
$$
\int_{\tilde{\gamma}} \frac{\frac{d\zeta}{dz}(z)dz}{\zeta(z)-w_0'}
$$ 
depends continuously on $w_0'$ and is always an integer. By making $w_0'$ tend to $w_0$ we see that 
$$
1=\# \zeta^{-1}(w_0)=\int_{\tilde{\gamma}} \frac{\frac{d\zeta}{dz}(z)dz}{\zeta(z)-w_0'}.
$$
and conclude $\# \zeta^{-1}(w_0')=1$ which yields $\# \bz^{-1}(w_0')=1$ as $\zeta$ takes values in the lower half plane on $P^0_1$. The degree of $\bz$ is $1$ and we are done. 
\end{proof}
\subsection{Surjectivity of $\bz$}
$(P, J_r)$ cannot be biholomorphic to $S^2$ as it isn't compact nor to $\bbC$ as it admits a positive harmonic function. By the uniformization theorem there is a holomorphic map $\kappa:(P,J_r)\rightarrow (\bbH,J_0).$ 

\begin{lemma} 
The map $\kappa$ is extendable to $\partial P.$ The extension is bijective.
\end{lemma}
We will denote the extension by $\tilde{\kappa}:\bar{P}\rightarrow \bar{\bbH}.$
\begin{proof}
Let $w\in \partial P$, $\epsilon>0$. Let $\mV_p$ be a neighbourhood of $p\in (\phi_\bbR^o)^{-1}(w)\subset X^o_\bbR$ and $\varphi_p:\bbD\rightarrow \mV_p$ a complex chart. Assume that $\varphi_p$ is actually defined in larger open set with image a neighbourhood of $\mV_p$ but we want to consider only its restriction to $\bbD$. The point $p$ is on the intersection of the closure of 2 open subsets of $X^o_\bbR$ say $P_0^0$ which we identify with $P$ and $P_0^1.$ Set $\bbD^+=\varphi_p^{-1}(\mV_p\cap P_0^0)$. Then
$$
\kappa \circ \phi_\bbR^o\circ\varphi_p:\bbD^+\rightarrow \kappa (\phi_\bbR^o(\mV_p\cap P_0^0)).
$$
Our first goal is to show $\bbD^+$ and $\kappa (\phi_\bbR^o(\mV_p\cap P_0^0))$ have boundaries which are Jordan curves.

\begin{itemize}
\item The boundary of $\bbD^+$ has two portions.
$$
\partial \bbD^+=\partial \bbD \cap \varphi_p^{-1}(\mV_p\cap P^0_0)\cup \varphi_p^{-1}(\partial P^0_0\cap \mV_p).
$$
The set $\partial \bbD \cap \varphi_p^{-1}(\mV_p\cap P^0_0)$ is a segment in $\partial \bbD$ as $\partial \mV_p\cap P_0^0$ is a Jordan segment and $\varphi_p^{-1}$ is a homeomorphism there. As for  $\varphi_p^{-1}(\partial P^0_0\cap \mV_p),$ this is a Jordan segment as $\partial P^0_0\cap \mV_p$ is.
\item The boundary of $\kappa (\phi_\bbR^o(\mV_p\cap P_0^0))$ also has two portions:
$$
\partial \kappa (\phi_\bbR^o(\mV_p\cap P_0^0))= \kappa (\phi_\bbR^o(\partial \mV_p\cap P_0^0)\cup s
$$
where $s$ is a subset of $\partial\bbH$ which we will show is an interval. Note that $\kappa (\phi_\bbR^o(\partial \mV_p\cap P_0^0))$ is a Jordan segment because $\partial \mV_p\cap P_0^0$ and $\kappa \circ \phi_\bbR^o$ is a homeomorphism on $P^0_0$. Showing that $s$ is an interval turns out to be rather technical.  The details are as follows.
\begin{lemma}
 The subset $s$ of $\partial \bbH$ defined above is an interval.
 \end{lemma}
 \begin{proof}
For the sake of completeness and because we will use the details here we recall Donaldson's proof the uniformization theorem from \cite{dRS}. Let $o\in P$ and consider a holomorphic coordinate $z$ around $o$ as well as a bump function around $o$ whose support is contained in a ball in the coordinate neighbourhood. We let $z_0$ denote the coordinate of $o$. Let
$$
A=\op\left(\frac{\beta}{z-z_0}\right).
$$  
This is a smooth $(0,1)$ form. There exists $h,$ a smooth function on $P$ such that $\p\op h=\p A$. The existence of such an $h$ is the core of the uniformization theorem and we simply make use of $h$. From proposition 32 in \cite{dRS} we know:
\begin{prop}[Donaldson]\label{prop32}
There is a sequence of smooth functions $h_i$ on $P$ such that
\begin{itemize}
\item $h_i=c_i$ a constant outside a compact set $B_i$ in $P.$\\
\item $||dh-dh_i||_\infty $ tends to zero as $i$ tends to $\infty$.\\ 
\end{itemize}
\end{prop}
Set $a=A-\op h+\overline{A-\op h}$ so that $da=0$ and let $\psi$ be a primitive of $a$ i.e. $d\psi=a.$ Consider as in \cite{dRS}
$$
F=\frac{\beta}{z-z_0}-h-\psi.
$$
This is a holomorphic map to $S^2$ by design and bijective onto $S^2\setminus I$ where $I$ is some closed interval in $\bbR$. We may as well assume that $I=[0,1].$ The map $\kappa$ is then given by
$$
\kappa(z)=\sqrt{1-\frac{1}{F(z)}}.
$$
Although as we will see $F$ extends by continuity to $\partial P$, the square root does not extend by continuity from $\bbC\setminus [0,+\infty[$ to $\bbC.$ So, it does not follow from the above that $\kappa$ extends to the boundary although we will eventually prove that in our case, i.e for $J_r$ on $P,$ it does.  Let $\{w_k\}$ be a sequence in $P$ tending to $w\in \partial P.$ We show that $\{F(w_k)\}$ converges. 
\begin{itemize}
\item First note that $dF$ is bounded in a neighbourhood of $w$.  In fact we may assume that $F=-h-\psi$ so that $dF=\p (\bar{h}-{h})$ which uniformly bounded from proposition (\ref{prop32})
$$|dF|\leq 2|dh|\leq 2\left(1+|dh_{i_0}|\right)$$ for $i_0$ sufficiently large which is bounded by a constant $K$ in a neighbourhood of $w$.
\item Next note
$$
F(w_{k'})-F(w_k)=F_1(1)-F_1(0)+i(F_2(1)-F_2(0))
$$
where $F_1$ and $F_2$ are the restrictions of the real and imaginary parts of $F$ to the segment $w_k+t(w'_k-w_k)$. By Taylor's theorem 
$$
F_1(t)-F_1(0)=F'_1(t_1)t, \quad F_2(t)-F_1(0)=F'_1(t_2)t
$$ 
for some $t_1$ and $t_2$ in $[0,1]$ and
$$
F'_1(1)=\Real\left(\frac{dF}{dz}\right)(w_{k'}-w_k), \quad F'_2(2)=\Imag\left(\frac{dF}{dz}\right)(w_{k'}-w_k)
$$
so that 
$$
|F'_1(t_1)|\leq K |w_{k'}-w_k|, \quad |F'_2(t_2)|\leq K|w_{k'}-w_k|
$$
as both $w_k+t_1(w_{k'}-w_k)$ and $w_k+t_2(w_{k'}-w_k)$ are in the small ball around $w$ where $dF$ is bounded by $K$. We thus have $|F(w_{k'})-F(w_k)|\leq K|w_{k'}-w_k|$ and $\{F(w_k)\}$ converges to an element in $S^2$. 
\end{itemize}
Because it is known that the inverse of $F$ is well defined and continuous on $S^2\setminus I$, this limit cannot be in $S^2\setminus I$ (otherwise this would force $w_k$ to converge to $F^{-1}$ of the limit). Hence we have constructed an extension $\tilde{F}: \bar{P}\rightarrow S^2$ of $F:P\rightarrow S^2\setminus I$. 

It follows from the fact that $F$ extends to the boundary of its domain as a continuous function that $\tilde{F} (\phi_\bbR^o( \mV_p)\cap \partial P)$ is an interval in $[0,1]$ say $[a,b]$ with $0<a<b<1$. In fact, because it is connected, $F (\phi_\bbR^o( \mV_p))$ must be on ``one side" of the interval $[0,1]$ in $S^2$ and 
$$
s=\left[\sqrt{1-\frac{1}{b}},\sqrt{1-\frac{1}{a}}\right] \text{or} \left[-\sqrt{1-\frac{1}{a}},-\sqrt{1-\frac{1}{b}}\right],
$$
so $s$ is a segment.
\end{proof}
\end{itemize}
We have thus proved that $\kappa \circ \phi_\bbR^o\circ\varphi_p:\bbD^+\rightarrow \kappa (\phi_\bbR^o(\mV_p\cap P_0^0))$ is an bijective map from two simply connected domains whose boundary is a Jordan curve.

Now it follows from Carath\'eodory's theorem that this map can be extended (as a homeomorphism) to the boundary. Because the point we were considering in $\partial P$ is arbitrary (the reasoning applies to vertices of $P$ with a small modification) we are done.


Let $\tilde{\kappa}:\bar{P}\rightarrow \bar{\bbH}$ be the extension. We proceed to show that it is injective. This is almost exactly as Carath\'eodory's original argument. Suppose the extension is not injective. Let $w$ and $v$ be two points in $\partial P$ such that $\tilde{\kappa}(w)=\tilde{\kappa}(v)$. Let $o$ be a point in the interior of $P$. Denote the segments $ow$ and $ov$ as $W$ and $V$ respectively. The Jordan curve $\kappa(W\cup V)$ goes through $\kappa(o)$ and $\tilde{\kappa}(w)=\tilde{\kappa}(v)$. Let $A$ be the interior of this curve and $B=\kappa^{-1}(A)$. Consider $Z$ a (maybe broken) segment in $\partial P$ connecting $w$ to $v$. Then
$$
\tilde{\kappa}(Z)\subset \partial A\cap \partial \bbH=\{\tilde{\kappa}(w)\}
$$
so that $\tilde{\kappa}$ is constant on $Z$. We may assume the constant is $1$. Now consider $\mV_p$ a neighbourhood of $p\in (\phi_\bbR^o)^{-1}(w)\subset X^o_\bbR,$ $\varphi_p:\bbD\rightarrow \mV_p$ a complex chart and
$$
{\kappa} \circ \phi_\bbR^o\circ\varphi_p:\bbD^+\rightarrow \kappa (\phi_\bbR^o(\mV_p\cap P_0^0)).
$$
where as before $\bbD^+=\varphi_p^{-1}(\mV_p\cap P_0^0)$. Now consider $u:\bbD\rightarrow \bbD^+$ given by Riemann's mapping theorem and 
$$
{\kappa} \circ \phi_\bbR^o\circ\varphi_p\circ u :\bbD \rightarrow \kappa (\phi_\bbR^o(\mV_p\cap P_0^0)).
$$
which by Carath\'eodory's theorem extends to $\partial \bbD$. The function $\phi_\bbR^o\circ\varphi_p\circ u$ also extends because by Carath\'eodory's theorem $u$ extends. Therefore, $\widetilde{(\phi_\bbR^o\circ\varphi_p\circ u)}^{-1}(Z_p),$ where $Z_p$ is the portion of $Z$ in $\phi_\bbR^o(\mV_p),$ and $\widetilde{(\phi_\bbR^o\circ\varphi_p\circ u)}$ denotes the extension of ${\phi_\bbR^o\circ\varphi_p\circ u},$   is actually an arc in $\bbD$. On this arc, the extension of ${\kappa \circ \phi_\bbR^o\circ\varphi_p\circ u}$ is constant and equal to 1. By Schwarts reflection principle it thus extends further to a neighbourhood of the given arc. But this arc has accumulation points and this brings about a contradiction as the zero sets of holomorphic functions are isolated in their domains of definition. We could similarly prove that the inverse of the extension is injective so the extension is bijective. 
\end{proof}

We are now in a position to prove the surjectivity of $\bz$. 
Let $\mU={z}(P).$ We have that
$$
\partial \mU=\partial\mU\cap \partial\bbH \cup \partial\mU\cap \bbH
$$ 
and we denote the second portion of the boundary as $\partial_0\mU=\partial\mU\cap \bbH.$ The goal is to show that $\partial_0\mU=\emptyset.$ Let $\bz_\kappa=\bz\circ \kappa^{-1}: \bbH\rightarrow \bbH.$ This map is holomorphic, injective and can be extended to the boundary $\partial P$ as a bijection $\partial P\rightarrow \bbH$ as both $\bz$ and $\kappa$ can be extended to the boundaries of their domains of definition. Let
$$
f(w)=\frac{1}{\bz\circ \kappa^{-1}\left(\frac{-1}{w}\right)}:\bbH \rightarrow \bbH.
$$
This map is holomorphic and maps $\bbR^*\subset \bbC$ to $\bbR$. By Schwarts reflection principle the formula $f(w)=\ov{f(\ov{w})}$ extends $f$ as a holomorphic function on $\bbC^*$ which we still denote by $f$ and which is still injective. The point $0$ is an isolated singularity of $f$.
\begin{itemize}
\item This singularity cannot be essential as by Picard's theorem $f$ could not be injective.
\item If the singularity is a pole, then it has to be a simple pole so as to not violate the injectivity of $f.$
\item The singularity may be removable.
\end{itemize}
But $\mU=\bz_\kappa(\bbH)$
The first thing to note is that
$$
\partial_0 \mU=\{\lim_k \bz_\kappa(w_{n_k}), \, (w_k)\, \text{unbounded}, \, (w_{n_k}) \,\text{subsequence of} \,(w_{k}) \}.
$$
This is because if $z\in \partial_0\mU$ then there is a sequence $(w_k)$ in $\bbH$ such that $\bz_\kappa(w_k) \rightarrow z.$ Assuming the sequence $(w_k)$ to be bounded would lead to a contradiction. Namely, if it were bounded it would have a convergent subsequence $(w_{n_k})$ in $\bar\bbH$ so that $\bz_\kappa(w_{n_k})\rightarrow \tilde{\bz}_\kappa(w)$ where $w$ denotes the limit of the considered subsequence $(w_{n_k})$ in $\bar{P}.$ If $w\in P$ then $\tilde{\bz}_\kappa(w)=\bz_\kappa(w)\in \mU$ but because $\mU$ is open this is incompatible with the assumption that $z\in \partial \mU$. If $w\in \partial P$, then $\tilde{\bz}_\kappa(w)\in \partial \bbH,$ which is incompatible with the assumption that $z\in \bbH$.  It follows that
$$
\partial_0 \mU=\{\lim_k f(w_{n_k}), \, \{w_k\}\subset \bbH, \, w_k\rightarrow 0, \, (w_{n_k}) \,\text{subsequence of} \,(w_{k}) \}.
$$
We have shown that $$f(0):=\{\lim_k f(w_{n_k}), \, \{w_k\}\subset \bbH, \, w_k\rightarrow 0, \, (w_{n_k}) \,\text{subsequence of} \,(w_{k}) \}$$ is either $\infty$ or a single point. Assume the latter. Let $a$ be this single point. Then $({\bz}_\kappa)^{-1}:\mU\rightarrow \bbC$ which is holomorphic has an isolated pole at $a.$ No such function can have image contained in $\bbH$ as one can see by considering the image of a neighbourhood of the pole.

\section{Uniqueness of the two families of scalar-flat K\"ahler toric metrics}
In this section the aim is to prove Theorem \ref{unique}. Note that the isothermal coordinates $(H,r)$ can also be described as satisfying
$$
u_{ij}dx_i\otimes dx_j=V\left(dH\otimes dH+dr\otimes dr\right),
$$
where $u_{ij}dx_i\otimes dx_j$ is the metric induced by $g$ on $X_\bbR$. We want to think of $(H,r)$ as a map of $(x_1,x_2).$ From what we showed in the previous section, this is a bijection from $P$ to $\bbH.$

\begin{proof}[Proof of Theorem \ref{unique}]
Without loss of generality we may assume that the unbounded edges of $X$ meet at the origin. The interior of the moment map image of $X$, $P$ is contained in the region
\[
\{x\in \bbR^2: x\ip \nu_1> 0,x\ip \nu_d >0 \}
\]
where $\nu_1$ and $\nu_d$ are the normals to the unbounded edges. Suppose we are given any K\"ahler toric metric on $X$ coming from a complete almost complex structure $J$. We can associate to $g$ the following quantities
\begin{itemize}
\item a symplectic potential $u$,
\item $\eta=(u_{x_1},u_{x_2})$. 
\end{itemize}
$X$ admits an ALE metric which we will denote by $g_{ALE}$ as before. In this case we write
\begin{itemize}
\item  $u_{ALE}$ for the symplectic potential of $g_{ALE}$,
\item and $\eta_{ALE}$ for the derivative of $u$ with respect to $(x_1,x_2)$. 
\end{itemize}

By reversing the construction in Section \ref{joyce_construction} we see that there are isothermal coordinates $(H,r)\in \bbH$ depending on $g$ and a map $\mu$ on $\bbH$ which gives the coordinates change between $(H,r)$ and symplectic coordinates $(x_1,x_2)$. We will sometimes refer to this map as the moment map for the torus action but it really is the moment map expressed in the coordinates $(H,r)$. This is the inverse of the map $\bz$ studied in the previous section and it follows from what we did there that $\mu$ is defined on the whole of $\bbH.$ There is also a function $\xi(H,r)=\eta \circ \mu(H,r)$ which is a solution of equation (\ref{PDE}) and can be seen as a harmonic, axi-symmetric function on $\bbR^3$. We will use the following notation:
\begin{itemize}
\item $u_0=u_{ALE}-u$,
\item $\eta_0=\eta_{ALE}-\eta$,
\item $\mu_0=\mu_{ALE}-\mu$,
\item and $\xi_0=\xi_{ALE}-\xi$. 
\end{itemize}
\begin{lemma}
There is a smooth function $f$ on $\bar{\bbH}$ such that 
\[
\mu_0=r^2f.
\]
\end{lemma}
\begin{proof}[proof of the lemma] We start by proving that $\mu_0$ extends as an analytic function to $\bbH$. From Proposition \ref{boundary_u} it follows that $u_0$ is smooth on a neighbourhood of $\partial P$ thus $\eta_0$ is bounded on a neighbourhood of each point in $\partial P$. We have
\[
D\mu_0=r\begin{pmatrix}
\frac{\partial\xi_{0,2}}{\partial r} & -\frac{\partial\xi_{0,2}}{\partial H} \\
-\frac{\partial\xi_{0,1}}{\partial r} &   \frac{\partial\xi_{0,1}}{\partial H}
\end{pmatrix}.
\] 
This follows from
 \begin{align}\nonumber
 &dx_1=r\left(\frac{\partial \xi_2}{\partial r} dH-\frac{\partial \xi_2}{\partial H} dr \right)\\ \nonumber
 &dx_2=-r\left(\frac{\partial \xi_1}{\partial r} dH-\frac{\partial \xi_1}{\partial H} dr \right),
 \end{align}
which holds for both $\mu$ and $\mu_{ALE}$. Thus, we need only to prove that $\xi_0$ extends as an analytic function to $\bar{\bbH}$. Since $\xi_0$ is harmonic outside the $H$ axis, a standard argument using the mean value property for harmonic functions ensures it is enough to see that $\xi_0$ is bounded in a neighbourhood of each point in the $H$ axis. We have $\xi_0=\eta_{ALE}\circ\mu_{ALE}-\eta\circ\mu$. This is the sum of the terms
\begin{enumerate}
\item $\eta_{ALE}\circ\mu_{ALE}-\eta_{ALE}\circ\mu$,
\item $\eta_{ALE}\circ\mu-\eta\circ\mu$.
\end{enumerate}
The second term is indeed bounded because $\eta_{ALE}-\eta$ is and $\mu$ extends as a continuous function to $\bar{\bbH}$. So what we need to do is show that the first term is bounded. 
We start by setting up some notation. Let $0<a_1<\cdots<a_{d-1}$ be positive real numbers (which are in fact determined by $P$ as explained in \cite{as}). Set $a_0=-\infty$ and $a_d=+\infty$. For $i=1,\cdots, d-1$ we define
\begin{itemize}
\item $H_i=H+a_i$,
\item  $\rho=\sqrt{H^2+r^2}$ and
\item $\rho_i=\sqrt{H_i^2+r^2}$.
\end{itemize}
Rewrite the first term as $\xi_{ALE}-\xi_{ALE}\circ (\mu_{ALE}^{-1}\circ\mu)$. We have an explicit expression for $\xi_{ALE}$ from \cite{as}. Namely
\begin{align}\nonumber
&\xi_{ALE,1} = \alpha_1\log(r)+\frac{1}{2}\sum_{i=1}^{d-1}(\alpha_{i+1}-\alpha_{i})\log\left( H_i+\rho_i\right)+\alpha H \\ \nonumber
&\xi_{ALE,2} = \beta_1\log(r)+\frac{1}{2}\sum_{i=1}^{d-1}(\beta_{i+1}-\beta_{i})\log\left( H_i+\rho_i\right)+\beta H.
\end{align}
Close to $\partial \bbH$, $\xi_{ALE}$ has a singularity and its singular behaviour is as follows
\[
H \in ]-a_{i+1},a_i[, r\rightarrow 0 \implies \xi_{ALE}= \nu_i \log r+O(1).
\]
This can be checked using the above formula but in fact, this needs to hold for any scalar-flat K\"ahler toric metric. Since we will use this fact, we explain why it holds. The symplectic potential $u$ satisfies $u=l_i\log l_i+O(1)$ near the $i$th facet of the moment map image. This implies that $\eta=\nu_i\log l_i+O(1)$ near that facet. But we have 
\[
r^2=\det (\Hess u)^{-1}=\gamma \prod_{i=1}^d l_i.
\]
The Guillemin boundary condition for $u$ shows that there is a nowhere vanishing smooth function $\gamma$ defined near $\partial P$ such that 
\[
(\det \Hess u)^{-1}=\gamma \prod_{i=1}^d l_i ,
\]
 and thus, near the interior of the $i$th facet of $P$, $\log l_i=\log r^2+O(1)$, where $O(1)$ denotes a function that is bounded in a neighbourhood of every point in $\partial P$. It follows that
\[
 H\in]-a_{i+1},a_i[, r\rightarrow 0 \implies \xi=\nu_i \log r+O(1).
\] 
We see that in order to show that $\xi_{ALE}-\xi_{ALE}\circ (\mu_{ALE}^{-1}\circ\mu)$ remains bounded as $r$ tends to zero we need to show two things
\begin{enumerate}
\item \label{H_relation} We need to show
\[
H\in]-a_{i+1},a_i[ \iff H'\in]-a_{i+1},a_i[
\]
where $H'$ is the $H$ coordinate of the map $\mu^{-1}={\bf z}$.
\item \label{r_relation} We also need to show that
\[
\log\frac{r}{r(\mu_{ALE}^{-1}\circ\mu)}
\]
is bounded. 
\end{enumerate}
First we prove (\ref{H_relation}). Let $a'_i$ be such that $\mu(a'_i,0)$ is the $i$th vertex of $P$ for $i=1,\cdots, {d-1}$. What we want to show is that $a'_i=a_i$ for all $i=1,\cdots, d-1$.  We may assume that $a_1=a'_1$. We claim that $a'_{i+1}-a'_i$ does not depend on the metric $g$. It is in fact given by
\[
a'_{i+1}-a'_i=\frac{\text{length}(e_{i+1})}{2\pi|\nu_i|^2},\,\,\, i=1,\cdots,d-1
\]
where $\text{length}(e_{i+1})$ is the length of the $i+1$ edge of $P$. We prove the following lemma:
\begin{lemma}
Consider any toric metric on $X$ and let $(H,r)$ be given as before by the map $\mu^{-1}={\bf z} $. Let $(a_i,0)=\mu^{-1}(p_i)$ where $p_i$ is the $i$-the vertex of $P$, for $i=1,\cdots, {d-1}$. Then
\[
a_{i+1}-a_i=\frac{\text{length}(e_{i+1})}{2\pi|\nu_i|^2},\,\,\, i=1,\cdots,d-1
\]
where $\text{length}(e_{i+1})$ is the length of the $i+1$ edge of $P$.
\end{lemma}

\begin{proof} Let $E_i$ be the pre-image via the moment map in $X$ of the $i$-th edge of $P$. For $i=2,\cdots,d-1$, the set $E_i$ is an $S^2$ and its volume is precisely the length of $e_i$. On the other hand we can calculate this volume using the $(H,r)$ coordinates. In fact
\[
\text{vol}(E_i)=\int_{E_i} \omega
\]
and 
\[
\omega=dx_1\wedge d\theta_1+dx_2\wedge d\theta_2
\]
where $(\theta_1,\theta_2)$ are coordinates in $\bbT^2$. We can rewrite $dx_1$ and $dx_2$ in $(H,r)$ coordinates and it follows that
\[
\omega=r\left(\frac{\partial \xi_2}{\partial r} dH-\frac{\partial \xi_2}{\partial H} dr \right)\wedge d\theta_1-r\left(\frac{\partial \xi_1}{\partial r} dH-\frac{\partial \xi_1}{\partial H} dr \right)\wedge d\theta_2.
\]
In these coordinates we see that 
\[
E_i=\{(H,r,\theta_1,\theta_2): r=0, H\in]-a'_{i+1},a'_i[, (\theta_1,\theta_2)\ip \nu_i=0\}
\]
Hence restricted to $E_i$, $\omega$ becomes
\[
\omega_{|E_i}=r\frac{\partial \xi_2}{\partial r} dH\wedge d\theta_1-r\frac{\partial \xi_1}{\partial r} dH \wedge d\theta_2.
\]
Now we can use the fact that 
\[
H\in]-a'_{i+1},a'_i[, r\rightarrow 0 \implies \xi=\nu_i \log r+O(1), 
\]
to rewrite the above as 
\[
\omega_{|E_i}=\alpha_idH\wedge d\theta_1-\beta_i dH \wedge d\theta_2.
\]
where $\nu_i=(\beta_i,\alpha_i)$. The direction in $\bbT^2$ which is not fixed in $E_i$ is the direction perpendicular to $\nu_i$ so that we write $(\theta_1,\theta_2)=t(\alpha_i,-\beta_i)$. Substituting in $\omega_{|E_i}$ we see that
\[
\omega_{|E_i}=|\nu_i|^2dH\wedge dt.
\]
Integrating the above over $]-a'_{i+1},a'_i[\times ]0,2\pi[$ yields the desired result.
\end{proof}

Now we proceed to prove (\ref{r_relation}). Composing with $\mu^{-1}$ we see that we may instead show that  
\[
\log\frac{r\circ\mu^{-1}}{r\circ\mu^{-1}_{ALE}}
\]
is bounded. As we have seen, for both metrics $g_{ALE}$ and $g$, the $r$ factor in $\mu_{ALE}^{-1}$ and $\mu^{-1}$ are given by 
\[
r\circ\mu^{-1}=\gamma \prod_{i=1}^d l_i^{1/2},
\]
and 
\[
r\circ \mu^{-1}_{ALE}=\gamma_{ALE} \prod_{i=1}^d l_i^{1/2},
\]
respectively, for some smooth nowhere vanishing functions $\gamma$ and $\gamma_{ALE}$ in a neighbourhood of $\partial P$. Thus our function is  
\[
\log\frac{\gamma}{\gamma_{ALE}}
\]
which is clearly bounded as $\gamma_g$ and $\gamma_{ALE}$ are nowhere vanishing.
\end{proof}
The function $f$ defined in the above lemma satisfies a linear PDE. Indeed it turns out that $f$ is an axi-symmetric harmonic function on $\bbR^5$. By this we mean that $f$ is a harmonic function on $\bbR^5$ which only depends on $H$ and the distance to the $H$ axis, $r$, where $H$ is one of the coordinates in $\bbR^5$.
\begin{lemma}[Wright]
Let $f$ be such that $\mu=\mu_{ALE}-r^2f$, then
\[
f_{HH}+f_{rr}+\frac{3f_r}{r}=0.
\]
\end{lemma}
\begin{proof}

Since
\[
f=\frac{\mu_0}{r^2}
\]
we have
\[
f_r=\frac{\mu_{0,r}}{r^2}-\frac{2\mu_0}{r^3}
\]
and 
\[
f_{rr}=\frac{\mu_{0,rr}}{r^2}-\frac{4\mu_{0,r}}{r^3}+\frac{6\mu_0}{r^4}.
\]
Also 
\[
f_{HH}=\frac{\mu_{0,HH}}{r^2}.
\]
We see that
\begin{equation}\label{f_intermediate}
f_{HH}+f_{rr}+\frac{3f_r}{r}=\frac{1}{r^2}\left( \mu_{0,HH}+\mu_{0,rr}-\frac{\mu_{0,r}}{r} \right).
\end{equation}
But we also have 
\[
\mu_{0,H}=r(\xi_{0,2,r},-\xi_{0,1,r})
\]
and
\[
\mu_{0,r}=-r(\xi_{0,2,H},-\xi_{0,1,H})
\]
therefore 
\[
\mu_{0,HH}=r(\xi_{0,2,Hr},-\xi_{0,1,Hr})
\]
and
\[
\mu_{0,rr}=-(\xi_{0,2,H},-\xi_{0,1,H})-r(\xi_{0,2,Hr},-\xi_{0,1,Hr}).
\]
Replacing in equation (\ref{f_intermediate}) the result of the lemma follows.
\end{proof}
Because $\mu(X)=P$ we must have $\mu\ip \nu_1 \geq 0$ and therefore defining $f_1=f\ip \nu_1$, we have
\[
f_1 \leq \frac{\mu_{ALE}\ip \nu_1}{r^2}.
\]
Now $\mu_{ALE}$ was explicitly calculated in \cite{as} and it is easy to see that $|\mu_{ALE}|\leq C\sqrt{H^2+r^2}$. It follows that there is a constant $C$ such that for any $w\in \bbR^5$,
\[
f_1(w)\leq \frac{C|w|}{r^2}.
\] 
Since $f_1$ is harmonic on $\bbR^5$ the mean value theorem states that,
\[
f_1(z)=\frac{1}{\text{vol}(\partial B(z,R))}\int_{\partial B(z,R)}f_1(w)dw.
\]
Here $B(z,R)$ denotes the ball with center $z$ and radius $R$ in $\bbR^5$ and $dw$ the induced euclidean volume on $\partial B(z,R)$. Therefore 
\[
f_1(z)=\frac{1}{AR^4}\int_{\partial B(z,R)}f_1(w)dw.
\]
for a universal constant $A$. We have  
\[
f_1(w)\leq C\frac{|w|}{r^2}\leq C'\frac{R}{r^2},
\]
where we think of $r$ as the distance to the $H$ axis in $\bbR^5$. Therefore 
\[
f_1(z)\leq \frac{C''}{R^3}\int_{\partial B(z,R)}\frac{dw}{r^2}.
\]
When $R$ is very large, the integral
\[
\int_{\partial B(z,R)}\frac{dw}{r^2}
\]
is of the same order of magnitude as 
\[
\int_{\partial B(0,R)}\frac{dw}{r^2}
\]
and a straightforward calculation shows that this integral is $O(R^2)$. 
\begin{lemma}
Let $B(R)$ be a the ball of radius $R$ in $\bbR^5$. Then there is a constant $C$ such that
\[
\int_{\partial B(R)}\frac{1}{r^2}dw\leq CR^2.
\]
\end{lemma}
\begin{proof}
A straightforward but perhaps tedious way to check the above is to parametrize the sphere of radius $R$ in $\bbR^5$. This can be done in the following way. Use the parameters $(\alpha,\alpha_1,\alpha_2,\alpha_3)$,
\[
\phi(\alpha,\alpha_1,\alpha_2,\alpha_3)=\begin{pmatrix}
                                                                 &R\cos \al \\
                                                                 &R\sin \al \cos \al_1\\\nonumber
                                                                 &R\sin \al \sin \al_1\cos \al_2\\ \nonumber
                                                                &R\sin \al \sin \al_1\sin \al_2\cos \al_3\\ \nonumber
                                                                &R\sin \al \sin \al_1\sin \al_2\sin \al_3\sin \al_3\\ \nonumber 
                                                                \end{pmatrix}
\]
In this parametrisation $r=R\sin \al$ and the integral becomes
\[
\int_{\mathcal{R}} \frac{R^4 \sin^3 \al \sin^2 \al_1 \sin \al_2}{R^2 \sin^2 \al}d\alpha d\alpha_1 d\alpha_2 d\alpha_3.
\]
on some bounded region $\mathcal{R}.$ The result follows.
\end{proof}
This then implies that $f_1$ must be bounded from above. Any bounded harmonic function on $\bbR^5$ is constant thus $f_1$ is constant. In the same way we show that $f\ip \nu_2$ is constant and since $\nu_1$ and $\nu_2$ are linearly independent we see that $f$ must be constant which in turn implies that
\[
\mu_0=r^2 v
\]
for some constant vector $v$ and given the relation between $\mu$ and $\xi$ this implies that
\[
\xi_0=H\nu
\]
where $\nu$ is a constant vector. Therefore either the metric $g$ is the ALE metric if $\nu=0$ or it is a Donaldson generalised Taub-NUT metric. 

It only remains to see that the constant vector $\nu$ must be in the cone determined by $\nu_1$ and $-\nu_d$. Write $\nu=(\al,\beta)$. Then 
\[
\mu=\mu_{ALE}+r^2\begin{pmatrix}
                                            -\beta\\
                                            \alpha
                                 \end{pmatrix}.           
\]
As we have seen $|\mu_{ALE}|\leq C \sqrt{H^2+r^2}$. Now fix a given $H$ and make $r$ tend to infinity we see that
\[
\mu\simeq r^2\begin{pmatrix}
                                            -\beta\\
                                            \alpha
                                 \end{pmatrix}.           
\]
as $r$ tends to infinity and therefore 
\[
\mu\ip \nu_1 \simeq-r^2\det (\nu_1,\nu).
\]
The fact that $\mu \ip \nu_1 > 0$ in the interior of $P$ implies that $\det(\nu,\nu_1)> 0$. In the same way we see that $\det(\nu,\nu_d)> 0$ and we are done.
\end{proof}

\section{Uniqueness given asymptotic behaviour}
The aim of this section is to show that the asymptotic behaviour of a scalar-flat K\"ahler toric metric on an unbounded toric manifold which is compatible with a complete toric complex structure determines the metric completely. 

We start by setting up some notation. Let $(X,\omega)$ be a strictly unbounded toric $4$-manifold with moment polytope $P$. Let $P_0$ be the convex polygonal region with only two unbounded sides coinciding with the unbounded sides of $P$ i.e. $P_0$ has only two normal vectors $\nu_1$ and $\nu_d$. We will assume further that these two unbounded edges meet at $0$ in $P_0$.

The convex set $P_0$ is the moment polytope of a toric orbifold $X_0$. This orbifold is of the form $\bbC^2/\Gamma$ for $\Gamma$ some finite subgroup of $U(2)$. We can endow $X$ and $X_0$ with K\"ahler structures. The manifold $X$ is a resolution of $X_0$ and there is a holomorphic map $\pi:X\rightarrow X_0$ which is $\bbT^2$ invariant. Outside a compact set in $X$ this map is the identity in action-angle coordinates.

In \cite{as} the authors show that for any $\nu\in \bbR^2$ satisfying
\begin{equation} \label{eq:condnu}
\det(\nu_1,\nu),\,\,\det(\nu_d,\nu)>0.
\end{equation}
there is a metric scalar-flat K\"ahler toric metric on $X$, the Donaldson generalised Taub-NUT metric with parameter $\nu$. We will denote this metric by $g^P_\nu$. The same parameter value $\nu$ defines a Donaldson generalised Taub-NUT orbifold metric with parameter $\nu$ on $X_0$ which we denote by $g^{P_0}_\nu$.

\begin{thm}\label{prescibed2}
Let $(X,\omega)$ be a strictly unbounded toric $4$-manifold endowed with a complete, compatible complex structure $J$ which is torus invariant.  Let $g$ be the corresponding toric metric determined by $\omega$ and $J$. The manifold $X$ is the resolution of toric compact orbifold $X_0=\bbC^2/\Gamma$.  Let $\pi:X\rightarrow X_0$ be the resolution map. If $g$ is scalar-flat and
\[
\left(\pi_*g-g^{P_0}_{\nu_0}\right)_{|P_0}(x)\rightarrow 0, \,\,\, x\rightarrow \infty
\]
then $g=g^P_{\nu_0}$.
\end{thm}
\begin{proof}
We know from theorem (\ref{unique}) that there is some $\nu\in \bbR^2$ satisfying condition \ref{eq:condnu} such that $g=g^P_{\nu}$. What we need to show is that $\nu=\nu_0$.

There is a function $u$ on $P$, the symplectic potential of $g$, such that $g$ is given by
\[
g=\begin{bmatrix}
\phantom{-}\Hess(u) & \vdots & 0\  \\
\hdotsfor{3} \\
\phantom{-}0 & \vdots & \Hess^{-1}(u)
\end{bmatrix}
\]
in action-angle coordinates. There is also a symplectic potential $u_0$ for the metric $g_{\nu_0}^{P_0}$,
\[
g_{\nu_0}^{P_0}=\begin{bmatrix}
\phantom{-}\Hess(u_0) & \vdots & 0\  \\
\hdotsfor{3} \\
\phantom{-}0 & \vdots & \Hess^{-1}(u_0)
\end{bmatrix}.
\]
Because the map $\pi$ is the identity outside a compact set the condition in the theorem says that 
\[
\Hess u-\Hess u_0\rightarrow 0
\]
at infinity in the polytope $P_0$.
\begin{lemma}\label{Hess}
Let $g$ be a K\"ahler toric metric on a symplectic toric $4$-fold $X$ with symplectic potential $u$. Let $(H,r)\in \bbH$ be isothermal coordinates for $g$ and $\mu:\bbH \rightarrow P$ be the coordinate change map and $\xi=(u_{x_1},u_{x_2})\circ \mu$ be defined as before. Then 
\[
\Hess u(x)=\frac{D\xi D\xi^{t}}{V}(\mu^{-1}(x))
\]
where $V=r\det{D\xi}$.
\end{lemma}
\begin{proof}
Let $\eta=(u_{x_1},u_{x_2})$ so that $\eta=\xi\circ\mu^{-1}$. We have
\[
\Hess u=D\eta=D\xi D\mu^{-1}
\]
Now as we have seen before
\[
D\mu=r\begin{pmatrix}
\xi_{2,r} & -\xi_{2,H}  \\
-\xi_{1,r} & \xi_{1,H}  
\end{pmatrix},
\] 
hence 
\[
D\mu^{-1}=\frac{D\xi ^t}{r\det D\xi}
\]
and the result follows.
\end{proof}

The inverse of Donaldson's version of Joyce's construction described in section \ref{joyce_construction} associates to $g$ isothermal coordinates $(H,r)$ and $\xi$ an axi-symmetric harmonic function $\bbR^3$ which we think of as a function in $\bbH$. We also have a coordinate change map $\mu:\bbH\rightarrow P$. Because $g=g^P_{\nu}$, we have an exact expression for $\xi$ from \cite{as}. We want to use that expression to study the asymptotic behaviour of $\Hess u$. We will prove the following lemma
\begin{lemma}
Let $g_\nu^P$ be a Donaldson generalised Taub-NUT metric on a symplectic toric $4$-fold $X$ with symplectic potential $u$. Let $(H,r)\in \bbH$ be isothermal coordinates for $g_\nu^P$. Then if $x$ tends to infinity in $P$ with $r(x)\rightarrow \infty$
\begin{itemize}
\item If $\nu\ne 0$ then 
\[
\Hess u(x)\rightarrow \frac{\nu\nu^t}{\det(v,\nu)},
\]
for some vector $v$.
\item If $\nu=0$ then 
\[
\Hess u(x)\rightarrow a\nu_1\nu_1^t,
\]
for some function $a$ and $a$ can take the value infinity in which case the above means that $|\Hess u|$ tends to infinity.
\end{itemize}
\end{lemma} 
\begin{proof}
Let $(H,r)\in \bbH$ be isothermal coordinates for $g_\nu^P$ and $\mu:\bbH \rightarrow P$ be the coordinate change map. 

We start with the $\nu \ne 0$ case. We will show that in this case
\[
\Hess u(x)=\frac{\nu\nu^t}{\det(v,\nu)}+\frac{A}{r(\mu^{-1}(x))}+O\left(\frac{1}{\rho^2(\mu^{-1}(x))}\right).
\]
where $A$ is a bounded matrix valued function on $P$ defined away from a compact set and $v$ is some vector. The result follows immediately from this.
From \cite{as} we have an explicit formula for the function  $\xi=(u_{x_1},u_{x_2})\circ \mu$. Namely
\begin{displaymath}
\xi_1 = \alpha_1\log(r)+\frac{1}{2}\sum_{i=1}^{d-1}(\alpha_{i+1}-\alpha_{i})\log\left( H_i+\rho_i\right)+\alpha H
\end{displaymath}
and
\begin{displaymath}
\xi_2 = \beta_1\log(r)+\frac{1}{2}\sum_{i=1}^{d-1}(\beta_{i+1}-\beta_{i})\log\left( H_i+\rho_i\right)+\beta H,
\end{displaymath}
where $\nu=(\alpha,\beta)$ and we use the same notation as before for $H_i$ and $\rho_i$. This implies
\[
D\xi=\begin{pmatrix}
\frac{\alpha_1}{r}+\sum_{i=1}^{d-1}\frac{(\alpha_{i+1}-\alpha_{i})r}{2\left( H_i+\rho_i\right)\rho_i} &\sum_{i=1}^{d-1}\frac{(\alpha_{i+1}-\alpha_{i})}{2\rho_i}+\alpha \\
\frac{\beta_1}{r}+\sum_{i=1}^{d-1}\frac{(\beta_{i+1}-\beta_{i})r}{2\left( H_i+\rho_i \right)\rho_i} &\sum_{i=1}^{d-1}\frac{(\beta_{i+1}-\beta_{i})}{2\rho_i}+\beta
\end{pmatrix}.
\]     
The map $\mu$ is proper and therefore when $x$ tends to infinity in $P$, $\rho$ also tends to infinity. Because
\begin{displaymath}
\frac{1}{ H_i+\rho_i}-\frac{1}{H+\rho}=O\left(\frac{1}{\rho^2}\right)
\quad
\text{and}
\quad
\frac{1}{\rho_i}-\frac{1}{\rho}=O\left(\frac{1}{\rho^2}\right),
\end{displaymath}
$D\xi$ can be written as
\[
\begin{pmatrix}
0&\alpha\\
0&\beta \\
\end{pmatrix}+
\frac{1}{r}\begin{pmatrix}
\alpha_1+\frac{(\alpha_d-\alpha_1)r^2}{2\rho(H+\rho)}
&\frac{(\alpha_d-\alpha_1)r}{2\rho}   \\
\beta_1+\frac{(\beta_d-\beta_1)r^2}{2\rho(H+\rho)}
&\frac{(\beta_d-\beta_1)r}{2\rho}
\end{pmatrix}
+ \frac{r}{\rho}O\left(\frac{1}{\rho^2}\right).
\]
Using the above, we can also write an expression for $V=r\det D\xi$ and study its asymptotic behaviour. 
\[
V=\det(\nu_1,\nu)\left(1-\frac{r^2}{2\rho(H+\rho)}\right)+\frac{\det(\nu_d,\nu)r^2}{2\rho(H+\rho)}+ \frac{\det(\nu_1,\nu_d)}{2\rho} +O\left(\frac{1}{\rho^2}\right) .
\]
The upshot of the above expressions is that
\begin{displaymath}
0\leq\frac{r^2}{2\rho(H+\rho)}\leq \frac{1}{2}.
\end{displaymath}
So this quantity tends to some value $a$ and it follows that, if $\nu\ne 0$, $V$ tends to $\det(v,\nu)$ where $v=(1-a)\nu_1+a\nu_d$. The asymptotic expression for $\Hess u$ follows from 
\[
\Hess u(x)=\frac{D\xi D\xi^t}{V}(\mu^{-1}(x))
\]
and 
\[
\begin{pmatrix}
0&\alpha\\
0&\beta \\
\end{pmatrix}
\begin{pmatrix}
0&0 \\
\alpha&\beta\\
\end{pmatrix}=\nu\nu^t.
\]

In the case where $\nu=0$, $D\xi$ can be written as
\[
\frac{1}{r}\begin{pmatrix}
\alpha_1+\frac{(\alpha_d-\alpha_1)r^2}{2\rho(H+\rho)}
&\frac{(\alpha_d-\alpha_1)r}{2\rho}   \\
\beta_1+\frac{(\beta_d-\beta_1)r^2}{2\rho(H+\rho)}
&\frac{(\beta_d-\beta_1)r}{2\rho}
\end{pmatrix}
+ \frac{r}{\rho}O\left(\frac{1}{\rho^2}\right).
\]
and $V$ is asymptotic to
\[
 \frac{\det(\nu_1,\nu_d)}{2\rho}
\]
and therefore $\Hess u$ is given by
\[
\frac{\rho}{\det(\nu_1,\nu_d)r^2}\left((\nu_1+a(\nu_d-\nu_1))(\nu_1+a(\nu_d-\nu_1))^t+b^2((\nu_d-\nu_1)(\nu_d-\nu_1)^t\right)+O\left(\frac{1}{\rho^2}\right)
\]
where 
\[
a=\lim \frac{r^2}{2\rho(H+\rho)},\qquad b=\lim \frac{r}{2\rho}.
\]
There are two cases to consider. The first case is when $\rho/r$ remains bounded. In this case $\Hess u$ converges to zero because $\rho/r^2$ tends to zero and we are done. If $\rho/r$ tends to infinity then $a=b=0$  and the result also follows.
\end{proof}
We want to apply the above lemma to both $g^P_\nu$ and $g^{P_0}_{\nu_0}$. In order to do that we must first find a sequence of points $x$ tending to infinity in $P$ such that both $r(x)$ and $r_0(x)$ tend to infinity where $r$ and $r_0$ are the first isothermal coordinate for $g^P_\nu$ and $g^{P_0}_{\nu_0}$ respectively. Let $x$ tend to infinity in $P$ with $r(x) \rightarrow \infty$. One can easily see that
\[
\left|\det \Hess u-\det \Hess u_0\right |\leq \left| \Hess u- \Hess u_0\right ||\Hess u|
\] 
and since $\Hess u$ is bounded (it tends to $a\nu\nu^t$) it follows that $\left|\det \Hess u-\det \Hess u_0\right |$ tends to zero. Now we know from the Donaldson's version of Joyce's construction in action-angle coordinates that
\begin{align}\nonumber
r&=\left(\det \Hess u\right)^{-1/2}\\ \nonumber
r_0&=\left(\det \Hess u_0\right)^{-1/2}.\\ \nonumber
\end{align}
Hence we see that
\[
\left|\frac{1}{r^2}-\frac{1}{r_0^2}\right|\rightarrow 0
\]
and it follows that $r_0(x)\rightarrow \infty$. 

Assume first that both $\nu$ and $\nu_0$ are different from zero. For this choice for $x$ we see that
\[
\Hess u-\Hess u_0 \rightarrow \frac{\nu\nu^t}{\det(v,\nu)}-\frac{\nu_0\nu_0^t}{\det(v_0,\nu_0)}
\]
for some vectors $v$ and $v_0$ and our assumption implies that 
\[
\frac{\nu\nu^t}{\det(v,\nu)}=\frac{\nu_0\nu_0^t}{\det(v_0,\nu_0)}.
\]
Now if $\nu\nu^t=a\nu_0\nu_0^t$ for some constant $a$ then $\nu$ is either $\nu_0$ or $-\nu_0$. This is because we must have for all $w \in \bbR^2$
\[
w^t\nu\nu^tw=aw^t\nu_0\nu_0^tw
\]
and therefore 
\[
|w\ip\nu|^2=|w\ip \nu_0|^2, \,\,\,  \forall \,w \in \bbR^2.
\]
Therefore $\nu$ and $\nu_0$ are proportional. This implies that in fact $\nu=\nu_0$ or $\nu=-\nu_0$. The condition $\det(\nu_1,\nu),\det(\nu_d,\nu)>0$ which both $\nu$ and $\nu_0$ must satisfy implies the result.

Now suppose that  $\nu=0$ and $\nu_0\ne 0$ is zero. Then again we see that if $r(x)$ tends to infinity so does $r_0(x)$ (we can use the fact that $\Hess u_0$ is bounded). We know that in this case
\[
\Hess u(x)\rightarrow a\nu_1\nu_1^t.
\]
If $a$ is infinity we immediately get a contradiction because $\Hess u-\Hess u_0$ would tend to infinity. If $a$ is finite we would conclude that $\nu_0$ is proportional to $\nu_1$ which is impossible (as this would give a non-complete metric). In the same way we get a contradiction if $\nu \ne 0$ and $\nu_0=0$.  
\end{proof}
\begin{rem}
The $\nu_0=0$ case of the above theorem was treated in by Wright. He proved that if $g$ is ALE then $g$ it is indeed $g^P_0$. Note that ALE implies the condition in our theorem but our condition is weaker. 
\end{rem}

\end{document}